\pdfoutput=1
\RequirePackage{ifpdf}
\ifpdf % We are running pdfTeX in pdf mode
\documentclass[pdftex]{sigma}
\else
\documentclass{sigma}
\fi

\usepackage[all]{xy}

\def \C{\mathbb{C}}
\def \Z{\mathbb{Z}}
\def \R{\mathbb{R}}
\def \p{\mathbb{P}}
\def \Q{\mathbb{Q}}
\def \A{\mathcal{A}}
\def \D{\mathcal{D}}

\def \K{{\bf K}}
\def \G{{\bf G}}
\def \w{{\underline{w}_0}}
\def \Div{\operatorname{Div}}
\def \vol{\operatorname{vol}}
\def \Sym{\operatorname{Sym}}
\def \codim{\operatorname{codim}}

\numberwithin{equation}{section}

\newtheorem{Th}{Theorem}[section]

\newtheorem{Prop}[Th]{Proposition}
\newtheorem{Cor}[Th]{Corollary}

\theoremstyle{definition}
\newtheorem{Ex}[Th]{Example}
\newtheorem{Def}[Th]{Definition}
\newtheorem{Rem}[Th]{Remark}

\begin{document}

\allowdisplaybreaks

\newcommand{\arXivNumber}{1911.00118}

\renewcommand{\thefootnote}{}

\renewcommand{\PaperNumber}{016}

\FirstPageHeading

\ShortArticleName{Intersections of Hypersurfaces and Ring of Conditions of a Spherical Homogeneous Space}

\ArticleName{Intersections of Hypersurfaces and Ring\\ of Conditions of a Spherical Homogeneous Space\footnote{This paper is a~contribution to the Special Issue on Algebra, Topology, and Dynamics in Interaction in honor of Dmitry Fuchs. The full collection is available at \href{https://www.emis.de/journals/SIGMA/Fuchs.html}{https://www.emis.de/journals/SIGMA/Fuchs.html}}}

\Author{Kiumars KAVEH~$^\dag$ and Askold G.~KHOVANSKII~$^{\ddag\S}$}

\AuthorNameForHeading{K.~Kaveh and A.G.~Khovanskii}

\Address{$^\dag$~Department of Mathematics, University of Pittsburgh, Pittsburgh, PA, USA}
\EmailD{\href{mailto:kaveh@pitt.edu}{kaveh@pitt.edu}}

\Address{$^\ddag$~Department of Mathematics, University of Toronto, Toronto, Canada}
\EmailD{\href{mailto:askold@math.utoronto.ca}{askold@math.utoronto.ca}}

\Address{$^\S$~Moscow Independent University, Moscow, Russia}

\ArticleDates{Received November 04, 2019, in final form March 14, 2020; Published online March 20, 2020}

\Abstract{We prove a version of the BKK theorem for the ring of conditions of a spherical homogeneous space~$G/H$. We also introduce the notion of ring of complete intersections, firstly for a spherical homogeneous space and secondly for an arbitrary variety. Similarly to the ring of conditions of the torus, the ring of complete intersections of~$G/H$ admits a~description in terms of volumes of polytopes.}


\Classification{14M27; 14M25; 14M10}

\renewcommand{\thefootnote}{\arabic{footnote}}
\setcounter{footnote}{0}

\section{Introduction}
Let $\Gamma_1,\ldots,\Gamma_n$ be a set of $n$ hypersurfaces in the $n$-dimensional complex torus $(\mathbb C^*)^n$ defined by equations $P_1=0,\ldots, P_n = 0$ where $P_1,\ldots,P_n$ are generic Laurent polynomials with given Newton polyhedra $\Delta_1,\ldots,\Delta_n$. The Bernstein--Koushnirenko--Khovanskii theorem (or BKK theorem, see \cite{Bernstein, Khovanskii-genus, Koushnirenko}) claims that the intersection number of the hypersurfaces $\Gamma_1,\ldots,\Gamma_n$ is equal to the mixed volume of $\Delta_1, \ldots, \Delta_n$ multiplied by $n!$. In particular when $\Delta_1=\cdots=\Delta_n=\Delta$ this theo\-rem can be regarded as giving a formula for the degree of an ample divisor class in a~projective toric variety.

Spherical varieties are a generalization of toric varieties to actions of reductive groups. Let~$G$ be a~connected reductive algebraic group. A normal variety $X$ with an action of $G$ is called {\it spherical} if a Borel subgroup of $G$ has an open orbit. When $G$ is a torus we get the definition of a toric variety. On the other hand, spherical varieties are generalizations of flag varieties $G/P$ as well. Similarly to toric varieties, the geometry of spherical varieties and their orbit structure can be read off from combinatorial and convex geometric data of fans and convex polytopes.

The BKK theorem has been generalized to spherical varieties by B. Kazarnovskii, M. Brion and the authors of the present paper. In particular, this generalization provides a formula for the degree of an ample $G$-linearized line bundle $L$ on a projective spherical variety $X$ (see \cite{Brion, Kazarnovskii, KKh-reductive}). The formula is given as the integral of a certain explicit function over the moment polytope (also called the Kirwan polytope) of $(X, L)$. This formula can be equivalently expressed as volume of a larger polytope associated to $(X, L)$. It is usually referred to as a Newton--Okounkov polytope of $(X, L)$ (see Section \ref{sec-NO-polytope-sph-var}).

In \cite{KKh-Annals, LM}, the BKK theorem has been generalized to intersection numbers of generic members of linear systems on any irreducible algebraic variety. The Newton--Okounkov bodies play a~key role in this generalization. All generalizations mentioned above deal with the intersection numbers of hypersurfaces in an algebraic variety.

Following the original ideas of Schubert, in the early 1980s De Concini and Procesi deve\-lo\-ped an intersection theory for algebraic cycles (which are not necessarily hypersurfaces) in a~symmetric homogeneous space $X = G/H$ (see~\cite{DP}). Their intersection theory, named the {\it ring of conditions of~$X$}, can be automatically generalized to a spherical homogeneous space~$X$.

For an algebraic torus $(\C^*)^n$ and more generally a horospherical homogeneous space $X$, nice descriptions of the ring of conditions are known. Moreover, combinatorial descriptions of cohomology rings as well as rings of conditions of interesting subclasses of spherical varieties have been obtained via equviariant cohomology methods by several authors (see \cite{Bifet-DeConcini-Procesi,Littelmann-Procesi, Strickland}). In our opinion, it is unlikely that for an arbitrary spherical homogeneous space, one can find such nice and transparent descriptions for the ring of conditions in terms of combinatorial data.

The contributions of the present paper are the following:
\begin{itemize}\itemsep=0pt
\item[(1)] We give a version of the BKK theorem in the ring of conditions of a spherical homogeneous space $X$, namely a formula for the intersection numbers of hypersurfaces in the ring of conditions in terms of volumes of polytopes (Theorem~\ref{th-main}). This modified version is as transparent as the original BKK theorem.
\item[(2)] Along the way, we introduce the notion of ring of complete intersections of a spherical homogeneous space. Unlike the ring of conditions, this ring admits a~nice description in terms of volumes of polytopes (see Theorem~\ref{th-ring-complete-intersec-vol-poly}, also cf.~\cite{Kaveh-note}). In the case when $X$ is a~torus~$(\C^*)^n$, the ring of complete intersections is isomorphic to the ring of conditions.
\item[(3)] We also introduce the ring of complete intersections for any irreducible algebraic variety (not necessarily equipped with a group action) and give a description of this ring in terms of volumes of convex bodies (Section~\ref{sec-ring-complete-intersec-general}).
\end{itemize}

We would like to point out that there is a big difference between the construction/definition of the ring of complete intersections for a spherical homogeneous space (as in~(2) above) and the construction/definition of the ring of complete intersections for an arbitrary variety (as in~(3) above). The first ring is defined similarly  to the ring of conditions except that in the definition of a cycle one should be more careful and consider (strongly) transversal hypersurfaces, The second ring is much larger and has a very different nature. It is defined by a general algebra construction that associates a (Poincar\'e duality) algebra to a vector space equipped with an $n$-linear form (Section~\ref{sec-alg-construction}). Nevertheless, the ring of complete intersections of a spherical homogeneous space can be described using the same general algebra construction.

\section{Preliminaries on spherical varieties}
In the rest of the paper we will use the following notation about reductive groups:
\begin{itemize}\itemsep=0pt
\item[--] $G$ denotes a connected complex reductive algebraic group.
\item[--] $B$ a Borel subgroup of $G$ and $U$, $T$ the maximal unipotent subgroup and a maximal torus contained in $B$ respectively.
\item[--] $G/H$ denotes a spherical homogeneous space, we let $\dim(G/H) = n$.
\end{itemize}

We recall some basic background material about spherical varieties. A normal $G$-variety $X$ is called \emph{spherical} if a Borel subgroup (and hence any Borel subgroup) has a dense orbit.
If~$X$ is spherical, it has a finite number of $G$-orbits as well as a finite number of
$B$-orbits. Spherical varieties are a generalization of toric varieties for actions of reductive groups. Analogous to toric varieties, the geometry of spherical varieties
can be read off from associated convex polytopes and convex cones. For an overview of the theory of spherical varieties, we refer the reader to~\cite{Perrin}.

It is a well-known fact that if ${\mathcal L}$ is a $G$-linearized line bundle on a spherical variety, then the space of sections $H^0(X, {\mathcal L})$ is a multiplicity free $G$-module.
For a quasi-projective $G$-variety $X$, this is equivalent to~$X$ being spherical.

Some important examples of spherical varieites and spherical homogeneous spaces are the following:
\begin{itemize}\itemsep=0pt
\item[(1)] When $G$ is a torus, the spherical $G$-varieties are exactly toric varieties.
\item[(2)] By the Bruhat decomposition the flag variety $G/B$ and the partial flag varieties $G/P$ are spherical $G$-varieties.
\item[(3)] Let $G \times G$ act on $G$ from left and right. Then the stabilizer of the identity is $G_{{\rm diag}} = \{(g, g) \,|\, g \in G\}$. Thus $G$ can be identified with the homogeneous space $(G \times G) / G_{{\rm diag}}$. Also by the Bruhat decomposition, this is a spherical $(G \times G)$-homogeneous space.
\item[(4)] Let $\mathcal{Q}$ be the set of all smooth quadrics in $\p^n$. The group $G = {\rm PGL}(n+1,\C)$ acts transitively on $\mathcal{Q}$.
The stabilizer of the quadric $x_0^2 + \cdots + x_n^2 = 0$ (in the homogeneous coordinates) is $H = {\rm PO}(n+1, \C)$ and hence $\mathcal{Q}$ can be identified with the
homogeneous space ${\rm PGL}(n+1, \C) / {\rm PO}(n+1, \C)$. The subgroup ${\rm PO}(n+1,\C)$ is the fixed point set of the involution $g \mapsto (g^t)^{-1}$ of $G$ and hence
$\mathcal{Q}$ is a symmetric homogeneous space. In particular, $\mathcal{Q}$ is spherical.
The homogeneous space $\mathcal{Q}$ plays an important role in classical enumerative geometry (see~\cite{DP}).
\end{itemize}

Throughout the rest of the paper we will fix a spherical homogeneous space $G/H$.

\section{Good compactification theorem} \label{sec-good- compactification}
The ring of conditions of $X=G/H$ is a version of the Chow ring for a (usually not complete) spherical homogenous space~$X$. The existence of a good compactification plays a crucial role in this intersection theory. One can define the ring of conditions more geometrically considering algebraic cycles as stratified analytic varieties and using cohomology rings instead of Chow rings. In this section we recall preliminaries on transversal intersections of stratified varieties and will state the theorem on existence of a good compactification.

Let $Y$ be an algebraic variety. A {\it stratification} of $Y$ is a decomposition of $Y$ into a disjoint union of smooth algebraic subvarieties (possibly of different dimensions). Each algebraic variety~$Y$ admits the following canonical stratification: Let $Y=Y_0\supset Y_1\supset\dots \supset \varnothing$ be the decreasing set of subvarieties where $Y_{i+1}$ is the set of singular points in $Y_i$. The {\it canonical stratification} of~$Y$ is the partition $Y=\bigcup_i Y_i^0$ where $Y_i^0=Y_i\setminus Y_{i+1}.$

Two subvarieties $Y,Z$ of an ambient smooth variety $X$ are {\it transversal} if there are stratifications $Y=\bigcup_i Y_i^0$, $Z=\bigcup Z_j^0$ such that any pair $Y_i^0$, $Z_j^0$ of smooth subvarieties are transversal in~$X$. Similarly one can define transversality of several subvarieties.

Let $X = G/H$ be a homogeneous space and $Y$, $Z$ smooth subvarieties. The Kleiman transversality theorem (which is a version of the famous Thom transversality theorem) states that: \emph{for almost all $g\in G$ the varieties $Y$ and $g \cdot Z$ are transversal in $X$, i.e., the subset $G^0\subset G$ such that the subvarieties $Y$ and $gZ$ are transversal in $X$ contains a nonempty Zariski open in $G$}.

The next important result is due to De Concini and Procesi~\cite{DP}. They considered the case when $G/H$ is a symmetric homogeneous space but their proof, more or less without change, works for a spherical homogeneous space as well.
\begin{Th}[existence of good compactification] \label{th-exist-good-comp}
Let $X=G/H$ be a spherical homogeneous space. Let $Y \subset X$ be a subvariety. Then there exists a complete $G$-variety $M$ that contains $X$ as its open $G$-orbit and for each $G$-orbit $O\subset M$ intersecting $\overline{Y} \subset M$ we have $\codim(\overline{Y}\cap O) = \codim(O) + \codim(Y)$.
\end{Th}

Let us say that two subvarieties $Y$, $Z$ of $X$ are {\it strongly transversal} in $X$ if there is a complete $G$-variety $M\supset X$ such that: (1) $M$ is a good compactification for $Y$ and $Z$ and (2) for each orbit $O\subset M$ the intersections $\overline Y\cap O$ and $\overline Z\cap O$ of their closures $\overline Y$ and $\overline Z$ with $O$ are transversal in $O$. Similarly, one defines strong transversality of several subvarieties.

Using Kleiman's theorem and the good compactification theorem one can show for any subvarieties $Y,Z\subset X$ there is a Zariski open subset $G_0\subset G$ such that, for any $g \in G_0$, the varieties~$Y$ and $g \cdot Z$ are strongly transversal in~$X$.

\section[Ring of conditions of $G/H$]{Ring of conditions of $\boldsymbol{G/H}$} \label{sec-ring-of-conditions}

In this section we consider a variant of intersection theory for a spherical homogeneous space $G/H$ (which is often a non-complete variety) called the {\it ring of conditions} $\mathcal{R}(G/H)$. Similar to the Chow ring, the elements of $\mathcal{R}(G/H)$ are formal linear combinations of subvarieties in $G/H$ but considered up to a different and stronger equivalence. The definition of ring of conditions goes back to De Concini and Procesi in their fundamental paper \cite{DP}. They introduced it as a~natural ring in which one can study many classical problems from enumerative geometry (this is related to Hilbert's Fifteenth problem). They also showed that~$\mathcal{R}(G/H)$ can be realized as a~direct limit of Chow rings of all smooth equivariant compactifications of~$G/H$.

Consider the set $\mathcal{C}$ of {\it algebraic cycles} in~$G/H$. That is, every element of $\mathcal{C}$ is a formal linear combination $V = \sum_i a_i V_i$ with $a_i \in \Z$ and $V_i$ irreducible subvarieties. Clearly, with the formal addition operation of cycles, $\mathcal{C}$ is an abelian group. If all the subvarieties $V_i$ in $V$ have the same dimension $k$ we say that~$V$ is a~$k$-cycle. For $0 \leq k \leq n$, the subgroup of $k$-cycles is denoted by~$\mathcal{C}_k$. For a cycle $V = \sum_i a_i V_i$ and $g \in G$ we define $g \cdot V$ to be $\sum_i a_i(g \cdot V_i)$. A $0$-cycle is just a~formal linear combination of points. If $P = \sum_i a_i P_i$ is a $0$-cycle where the $P_i$ are points, we let $|P| = \sum_i a_i$.

Let $Y$ and $Z$ be strongly transversal irreducible subvarieties, and let $Y \cap Z$ be a union of irreducible components $T$. We then define the {\it intersection product} $Y \cdot Z$ to be the cycle
\[
Y \cdot Z= \sum_T T.
\]
By linearity the intersection product can be extended to algebraic cycles $Y = \sum_i a_i Y_i$ and $Z = \sum_j b_j Z_j$ such that for all~$i$,~$j$ varieties $Y_i$ and~$Z_j$ are strongly transversal to each other.

We now define an equivalence relation $\sim$ on the set of algebraic cycles as follows. Let $V=\sum_i a_i V_i$, $V'=\sum_j b_j V'_j \in \mathcal{C}_k$ be algebraic cycles of dimension $k$. Let $Z$ be an irreducible sub\-va\-riety of complementary dimension $n - k$. One knows that for generic $g \in G$, the subvariety~$g \cdot Z$ intersects all the $V_i$ and $V'_j$ in a strongly transversal fashion. Thus, for generic $g \in G$, the intersection products $V \cdot (g \cdot Z)$ and $V' \cdot (g \cdot Z)$ are defined and are $0$-cycles. We say that $V \sim V'$ if
for any $Z$ as above and for generic $g \in G$ we have:
\begin{gather} %\label{equ-numerical-equiv}
 |V \cap (g \cdot Z)| = |V' \cap (g \cdot Z)|.
 \end{gather}
That is, $V \sim V'$ if they intersect general translates of any subvariety of complementary dimension at the same number of points.

One verifies using good compactifications that if $Y_1 \sim Y_2$ and $Z_1 \sim Z_2$ then for generic $g\in G$ the intersection products $Y_1 \cdot gZ_1 $ and $Y_2 \cdot gZ_2$ are equivalent to each other. Thus the intersection product of strongly transversal subvarieties induces an intersection operation on the quotient~$\mathcal{C} /{\sim}$.

\begin{Def} \label{def-ring-of-conditions} The {\it ring of conditions of $G/H$} is $\mathcal{C} /{\sim}$ with the ring structure coming from addition and intersection product of cycles.
\end{Def}

The following example, due to De Concini and Procesi, shows the assumption that $G/H$ is spherical is important and the ring of conditions is not well-defined for all homogeneous spaces.
\begin{Ex} \label{ex-ring-of-cond-not-well-def}Take the $3$-dimensional affine space $\C^3$ regarded as an additive group. We would like to show that the ring of conditions of $\C^3$ does not exist. By contradiction suppose that it exists. Consider the surface (quadric) $S$ in $\C^3$ defined by the equation $y = zx$. The intersection of a horizontal plane $z=a$ and $S$ is the line $y=ax$. Thus for almost all $a\in \C$ the lines $y=ax$, $z=a$ must be equivalent in the ring of conditions of $\C^3$. On the other hand we claim that two skew lines $\ell_1$ and $\ell_2$ cannot be equivalent. This is because one can find a $2$-dimensional plane~$P$ such that any translate of~$P$ intersects~$\ell_1$ but no translate of~$P$ intersects~$\ell_2$ unless it contains~$\ell_2$. The contradiction shows that the ring of conditions of~$\C^3$ is not well-defined.
\end{Ex}

\section[Ring of complete intersections of $G/H$]{Ring of complete intersections of $\boldsymbol{G/H}$} \label{sec-ring-of-complete-intersec}

In this section we propose an analogue of the ring of conditions constructed using only (non-degenerate) complete intersections in a spherical homogeneous space~$G/H$.

\begin{Def} \label{def-non-degenerate-complete-intersection} A {\it non-degenerate complete intersection} in $G/H$ is an intersection of several strongly transversal hypersurfaces $\Gamma_1,\ldots, \Gamma_k\subset G/H$.
\end{Def}

Consider the collection $\mathcal{C}'_k$ of all $k$-dimensional {\it complete intersection cycles}, that is, formal linear combinations $\sum_i a_i V_i$ where each $V_i$ is a $k$-dimensional non-degenerate complete intersection in $G/H$. We let $\mathcal{C}' = \bigoplus_{i=0}^n \mathcal{C}'_k$.

As in the construction of ring of conditions (Section \ref{sec-ring-of-conditions}) we define an equivalence relation $\sim$ on $\mathcal{C}'$ as follows: let $V$, $V' \in \mathcal{C}'_k$ be complete intersection $k$-cycles. We say that $V \sim V'$ if for any complete intersection $(n-k)$-cycle $Z$ and generic $g \in G$ we have
$|V \cap (g \cdot Z)| = |V' \cap (g \cdot Z)|$ (recall $n = \dim(G/H)$).
Note that in defining $\sim$ we only use $Z$ that are complete intersections.

\begin{Def} \label{def-ring-of-complete-intersections}
The {\it ring of complete intersections $\mathcal{R}'(G/H)$} is
$\mathcal{C}' /{\sim}$ with the ring structure coming from addition and intersection product of complete intersection cycles.
\end{Def}

In the same fashion as for the ring of conditions, it can be verified that the ring $\mathcal{R}'(G/H)$ is well-defined.
The inclusion $\mathcal{C}' \subset \mathcal{C}$ induces a ring homomorphism $\mathcal{R}'(G/H) \to \mathcal{R}(G/H)$. In general, this homomorphism is neither injective nor surjective.

In Section \ref{sec-descp-ring-comp-int} we will give a description of $\mathcal{R}'(G/H)$ in terms of volumes of polytopes.

\section{Newton--Okounkov polytopes for spherical varieties} \label{sec-NO-polytope-sph-var}
In this section we recall the notion of a Newton--Okounkov polytope associated to a $G$-linear system on a spherical homogeneous space $G/H$. The volume of this polytope gives the self-intersection number of the linear system.

Let $X$ be an $n$-dimensional projective spherical $G$-variety with a $G$-linearized very ample line bundle ${\mathcal L}$. One associates a convex polytope $\Delta(X, {\mathcal L})$ to $(X, {\mathcal L})$. The construction depends on the combinatorial choice of a {\it reduced word decomposition} $\w$ for the longest element $w_0$ in the Weyl group of~$G$.

A main property of the Newton--Okounkov polytope $\Delta(X, {\mathcal L})$ is that its volume gives a~formula for the self-intersection number of divisor class of ${\mathcal L}$. Namely,
\begin{gather} \label{equ-deg}
c_1({\mathcal L})^n = n! \vol_n(\Delta(X, {\mathcal L})).
\end{gather}
The formula \eqref{equ-deg} is equivalent to the Brion--Kazarnovskii formula for the degree of a projective spherical variety (see \cite[Theorem~2.5]{Kaveh-note} as well as \cite{Kazarnovskii} and \cite{Brion}).

\begin{Rem} \label{rem-Brion-Kaz-original-statement}
The original versions of \eqref{equ-deg} in \cite{Kazarnovskii} and \cite{Brion}, do not express the answer as volume of a Newton--Okounkov polytope but rather as integral of a certain function over the {\it moment polytope} $\mu(X, {\mathcal L})$. The construction of $\Delta(X, {\mathcal L})$ appears in \cite{Okounkov-sph} and \cite{AB} motivated by a question of the second author.
\end{Rem}

\begin{Rem} \label{rem-not-normal}
If $X$ is not normal the polytope $\Delta(X, {\mathcal L})$ can still be defined using the integral closure of the ring of sections of ${\mathcal L}$ (in its field of fractions), see \cite[Section~6.2]{KKh-reductive}.
\end{Rem}

More generally, let $X$ be a not-necessarily-projective spherical variety and let $E$ be a $G$-linear system on~$X$, that is, $E$ is a finite-dimensional $G$-invariant subspace of $H^0(X, {\mathcal L})$ for a~$G$-linearized line bundle ${\mathcal L}$ on~$X$. Extending the notion of intersection number of divisors on complete varieties, one can talk about intersection number of linear systems. Let $E_1, \ldots, E_n$ be linear systems on $X$. The intersection index $[E_1, \ldots , E_n]$ is defined to be the number of solutions in~$X$ of a generic system of equations $f_1(x) = \cdots = f_n(x) = 0$, where $f_i \in E_i$. When counting the solutions, we ignore the solutions $x$ at which all the sections in some $E_i$ vanish. An important property of the intersection index is multi-additivity with respect to product of linear systems. In \cite[Section~6]{KKh-reductive} the authors define the Newton--Okounkov polytope $\Delta(E)$ and prove the following.

\begin{Prop} \label{prop-Brion-Kaz-linear-system}
The intersection index $[E, \ldots, E]$ is equal to $n! \vol_n(\Delta(E))$.
\end{Prop}

From the polarization formula in linear algebra the following readliy follows.
\begin{Cor} \label{prop-int-number-several-divisors}
Let ${\mathcal L}_1, \ldots, {\mathcal L}_n$ be $G$-linearized very ample line bundles on a spherical variety~$X$ and take $G$-invariant linear systems $E_i \subset H^0(X, {\mathcal L}_i)$. For any subset $I \subset \{1, \ldots, n\}$ let $\Delta_I = \Delta\big(\prod_{i \in I} E_i\big)$. We have the following formula for the intersection index of the~$E_i$:
\begin{gather*}
(-1)^n [E_1, \ldots, E_n] = - \sum_{i} \vol_n(\Delta_i) + \sum_{i < j} \vol_n(\Delta_{i,j}) + \cdots
+ (-1)^n \vol_n(\Delta_{1, \ldots, n}).
\end{gather*}
\end{Cor}

\begin{Rem} \label{rem-NO-body}The above notion of the Newton--Okounkov polytope of a $G$-linear system over a spherical variety is a special case of the more general notion of a~{\it Newton--Okounkov body} of a~graded linear system on an arbitrary variety (see \cite{KKh-Annals, LM} and the references therein).
\end{Rem}

Suppose $E$ is a very ample $G$-linear system on $X$. That is, the Kodaira map of $E$ gives an embedding of $X$ into the dual projective space $\p(E^*)$. Let $Y_E$ denote the closure of the image of $X$ in $\p(E^*)$ and let ${\mathcal L}$ be the $G$-linearized line bundle on $Y_E$ induced by $\mathcal{O}(1)$ on $\p(E^*)$. One then has $\Delta(E) = \Delta(Y_E, {\mathcal L})$. We note that in general~$Y_E$ may not be normal, nevertheless one can still define the polytope $\Delta(Y_E, {\mathcal L})$ (see Remark~\ref{rem-not-normal}).

One can show that the map $E \mapsto \Delta(E)$ is piecewise additive in the following sense. Let $C$ be a~rational polyhedral cone generated by a finite number of $G$-linear systems $E_1, \ldots, E_s$. Then the there is a rational polyhedral cone $\tilde{C}$ (living in some appropriate Euclidean space) and a~linear projection $\pi\colon \tilde{C} \to C$ such that for each $G$-linear system $E \in C$ we have $\Delta(E) = \pi^{-1}(E)$. It is then not hard to see that there is a fan $\Sigma$ supported on~$C$ such that the map
$E \mapsto \Delta(E)$ is additive when restricted to each cone in $\Sigma$ (see \cite[Proposition~1.4]{KV}).

When the line bundles belong to a cone of the above fan in which the Newton--Okounkov polytope is additive, the formula for their intersection number in Corollary \ref{prop-int-number-several-divisors} can be simplified to the mixed volume of their corresponding Newton--Okounkov polytopes.
\begin{Cor} \label{cor-mixed-vol-Brion-Kaz}
Suppose $E_1, \ldots, E_n$ are $G$-linear systems which lie in a cone on which the map $E \mapsto \Delta(E)$ is additive. We then have
\begin{gather*} E_1 \cdots E_n = n! V(\Delta(E_1), \ldots, \Delta(E_n)),\end{gather*}
where $V$ denotes the mixed volume of polytopes in $n$-dimensional Euclidean space.
\end{Cor}
\begin{proof}The corollary immediately follows from Proposition~\ref{prop-Brion-Kaz-linear-system}, additivity of the Newton--Okounkov polytope and the multi-linearity of mixed volume.
\end{proof}

\section{A version of BKK theorem for ring of conditions}\label{sec-BKK-spherical}
Proposition \ref{prop-Brion-Kaz-linear-system} and Corollary~\ref{cor-mixed-vol-Brion-Kaz} express the intersection numbers of generic elements of $G$-linear systems on $X=G/H$ in terms of volumes/mixed volumes of certain (virtual) polytopes. In this section we will modify this so that it computes the intersection numbers of hypersurfaces in the ring of conditions of $X$.

We have to deal with the following two problems: (1) In general an algebraic hypersurface in~$X$ is not a section of a $G$-linearized line bundle on $X$. (Although in some cases, for example for $X=G=(\C^*)^n$, any hypersurface in~$X$ is indeed a section of some $G$-linearized line bundle on~$X$.) (2) One has to show that it is possible to make any given finite collection of hypersurfaces ``generic enough'' by moving its members via generic elements of the group~$G$.

Fortunately, it is known how to overcome the first problem: any hypersurface after multiplication by a~natural number $m$ becomes a section of a $G$-linearized line bundle. The second problem is also not complicated to solve (see below).

Let $G/H$ be a spherical homogeneous space of dimension~$n$ and let~$D$ be an effective divisor (i.e., a linear combination of prime divisors with nonnegative coefficients) in~$G/H$. Let ${\mathcal L} = \mathcal{O}(D)$ be the corresponding line bundle on~$G/H$. We would like to associate a Newton--Okounkov polytope to the divisor~$D$ such that the volume of this polytope gives the self-intersection number of $D$ in the ring of conditions of~$G/H$. To this end, we need to equip~${\mathcal L}$ with a $G$-linearization. The following results are well-known (see \cite{KKLV, Popov}).

\begin{Th} \label{th-exist-G-lin}
Let $X$ be a normal $G$-variety and let $\mathcal{L}$ be a line bundle on~$X$. Then there is $m>0$ such that $\mathcal{L}^{\otimes m}$ has a $G$-linearization.
\end{Th}

Alternatively, there exists a finite covering $\tilde{G} \to G$ where $\tilde{G}$ is a connected reductive group with trivial Picard group. Moreover, after replacing $G$ with $\tilde{G}$, every line bundle on $X$ admits a linearization (see remark after Proposition 2.4 and Proposition 4.6 in \cite{KKLV}). Thus, every hypersurface is a section of a linearizable line bundle.

\begin{Th} \label{th-G-lin-G/H}
The assignment $\chi \mapsto \mathcal{L}_\chi$ gives a one-to-one correspondence between characters of $H$ and $G$-linearized line bundles $\mathcal{L}$ on $G/H$.
\end{Th}

Let $m>0$ be such that ${\mathcal L}^{\otimes m}$ has a $G$-linearization and fix a $G$-linearization for ${\mathcal L}^{\otimes m}$. Let $s_m \in H^0\big(G/H, {\mathcal L}^{\otimes m}\big)$ be the section defining $mD$, i.e., $mD = \operatorname{div}(s_m)$. Let $E_m \subset H^0\big(G/H, {\mathcal L}^{\otimes m}\big)$ be the $G$-invariant subspace generated by~$s_m$. Let us denote the Newton--Okounkov polytope associated to $E_m$ by $\Delta(mD)$ (see Section~\ref{sec-NO-polytope-sph-var}). We note that the linear system $E_m$ and the corresponding polytope $\Delta(mD)$ depend on the choice of linearization of~$G/H$. We put
\begin{gather*} \Delta(D) = (1/m) \Delta(mD).\end{gather*}

\begin{Th} \label{th-main}The self-intersection number of $D$ in the ring of conditions of~$G/H$ is \linebreak $n! \vol_n(\Delta(D))$.
\end{Th}
\begin{proof}
The theorem can be reduced to Proposition~\ref{prop-Brion-Kaz-linear-system} via Kleiman's transversality theorem and existence of a good compactification for~$D$ (Theorem~\ref{th-exist-good-comp}). Without loss of generality let $m=1$ and put $E=E_m$. Let $\overline{X}$ be a spherical embedding of $G/H$ which provides a good compactification for~$D$, i.e., for each $G$-orbit $O\subset \overline{X}$ such that $\overline{D}\cap O\neq \varnothing$ we have $\overline{D}\cap O$ is a hypersurface in~$O$. By Kleiman's transversality theorem if $g_1, \ldots, g_n \in G$ are in general position then for any $G$-orbit $W$ in $\overline{X}$ all the $\overline{g_i \cdot D}\cap W$ intersect strongly transversally in $W$. In particular if $W\neq G/H$ the intersections are empty. Thus the self-intersection number of~$D$ in the ring of conditions is equal to the self-intersection number of the divisors $\overline{g_i \cdot D}$ in $\overline{X}$.

On the other hand, we know the following: take $s_1, \ldots, s_n \in E$ and let $H_i = \{x \in G/H \,|\, s_i(x) \allowbreak = 0 \}$ with closure $\overline{H}_i \subset \overline{X}$. Suppose all the $\overline{H}_i$ intersect transversally. In particular, $\overline{H}_1 \cap \cdots \cap \overline{H}_n$ lies in $G/H$, that is, there are no intersection points at infinity. Then $[E, \ldots, E]$ is equal to $|H_1 \cap \cdots \cap H_n|$, the number of solutions $x \in G/H$ of the system $s_1(x) = \cdots = s_n(x) = 0$ (we note that since $E$ is $G$-invariant the set of common zeros of all the sections in~$E$ is empty).
It follows that if the $g_i \in G$ are in general position we have
\begin{gather*} [E, \ldots, E] = |\overline{g_1 \cdot D} \cap \cdots \cap \overline{g_n \cdot D}|.
\end{gather*}
Now since $G$ is connected, the action of $G$ on the Picard group of $\overline{X}$ is trivial. Hence for any $g \in G$, the divisor $\overline{g \cdot D}$ is linearly equivalent to~$\overline{D}$. Putting everything together we conclude that
\begin{gather*} [E, \ldots, E] = \overline{D}^n.\end{gather*}
 The theorem now follows from Proposition \ref{prop-Brion-Kaz-linear-system}.
\end{proof}

Let $D_1, \ldots, D_n$ be $n$ irreducible hypersurfaces in $G/H$. For each $i$, fix a $G$-linearization for the line bundle $\mathcal{O}(D_i)$. For each collection $I = \{i_1, \ldots, i_s\} \subset \{1, \ldots, n\}$ equip the line bundle $\mathcal{O}(D_{i_1} + \cdots + D_{i_s})$ with the $G$-linearization induced from those of the $D_i$. Let $\Delta_I$ be the Newton--Okounkov body $\Delta(D_{i_1} + \cdots + D_{i_s})$.
The following is an immediate corollary of Theorem \ref{th-main}.
\begin{Cor} \label{cor-main}
The intersection number of the $D_i$ in the ring of conditions is given by the following formula\begin{gather*}
(-1)^n D_1 \cdots D_n = - \sum_{i} \vol_n(\Delta_i) + \sum_{i < j} \vol_n(\Delta_{i,j}) + \cdots
+ (-1)^n \vol_n(\Delta_{1, \ldots, n}).
\end{gather*}
\end{Cor}

\section{Two algebra constructions} \label{sec-alg-construction} In this section we discuss two similar constructions of graded commutative algebras with Poincar\'{e} duality associated with a $\Q$-vector space $V$. The first construction uses a symmetric $n$-linear function $F$ on $V$. The second construction uses a homogeneous degree $n$ polynomial $P$ on $V$. The constructions produce isomorphic graded algebras with Poincar\'{e} duality if $F$ and $P$ satisfy $n! F(x,\dots,x)=P(x)$ for all $x \in V$. We will give the results without proofs as all the proofs are straightforward.
\subsection{Algebra constructed from a symmetric multi-linear function} Let $V$ be a (possibly infinite-dimensional)
vector space over $\Q$ equipped with a nonzero symmetric $n$-linear function $F\colon V \times \cdots \times V \to \Q$. To $(V, F)$ one can associate a graded $\Q$-algebra $A _F= A = \bigoplus_{i=0}^n A_i$ by the following construction. Let $\Sym(V)= \bigoplus_{i \geq 0} \Sym^i(V)$ be the symmetric algebra of the vector space $V$. With $F$ one can associate the function $F_s$ on $\Sym(V)$ as follows: (1)~if $a\in \Sym^n(V)$ and $a=\sum \lambda^iv_1^i\cdots v^i_n$ where $\lambda_i\in \Q$, $v^i_j\in V$ then $F_s(a)=\sum\lambda_i F\big(v_1^i, \ldots, v^i_n\big)$; (2)~if $a\in \Sym^k(V)$ where $k \neq n$ then $F_s(a)=0$. Let $I_{F_s}\subset \Sym(V)$ be the subset defined by $a\in I_{F_s}$ if and only if for any $b\in \Sym(V)$ the identity $F_s(ab)=0$ holds. It is easy to see that~$I_{F_s}$ is a homogeneous ideal in $\Sym(V)$. We define $A$ to be the graded algebra $\Sym(V) / I_{F_s}$. It follows from the construction that $A$ has the following properties:
\begin{itemize}\itemsep=0pt
\item $A_0=\Q$.
\item $A_k=0$ for $k>n$.
\item $\dim _{\Q}A_n=1$, moreover the function $F_s$ induces a linear isomorphism $f_s\colon A_n \rightarrow \Q$.
\item $A_1$ coincides with the image of $V=\Sym^1(V)$ in $A=\Sym(V)/I_{F_s}$.
\item $A$ is generated as an algebra by $A_0$ and $A_1$.
\item The pairing $B_k\colon A_k\times A_{n-k}\rightarrow \Q$ by formula $B_k(a_k, a_{n-k})=f_s(a_k a_{n-k})$ is non-degenerate. It provides a ``Poincar\'e duality'' on~$A$.
\end{itemize}

\subsection{Algebra constructed from a homogeneous polynomial}
Let $V$ be a (possibly infinite-dimensional) $\Q$-vector space equipped with a homogeneous deg\-ree~$n$ polynomial $P\colon V \to \Q$ (recall that a function $P\colon V \to \Q$ is a polynomial or a~homogeneous degree $n$ polynomial if its restriction to any finite-dimensional subspace $V_1\subset V$ is correspondingly a polynomial or a homogeneous degree~$n$ polynomial). To~$(V, P)$ one can associate a graded $\Q$-algebra $A_P = A = \bigoplus_{i=0}^n A_i$ as we explain below.

First we recall the algebra $\D = \D_V$ of constant coefficient differential operators on~$V$. For a~vector $v \in V$, let $L_v$ be the differentiation operator (Lie derivative) on the space of polynomial functions on $V$ defined as follows. Let $f$ be a polynomial function on~$V$, then
\begin{gather*} L_v(f)(x) = \lim_{t \to 0} \frac{f(x+tv) - f(x)}{t}.\end{gather*}
The algebra $\D$ is defined to be the subalgebra of the $\Q$-algebra of linear operators on the space of polynomials on~$V$ generated by the Lie derivatives~$L_v$ for all $v \in V$ and by multiplication by scalars $c \in \Q$. The algebra $\D$ is commutative since Lie derivatives against constant vector fields commute. The algebra $\D$ can be naturally identified with the symmetric algebra $\Sym(V)$. When $V \cong \Q^n$ is finite-dimensional, $\D$ can be realized as follows: Fix a basis for~$V$ and let $(x_1, \ldots, x_n)$ denote the coordinate functions with respect to this basis. Each element of $\D$ is then a polynomial expression, with constant coefficients, in the differential operators $\xi_1=\partial/\partial x_1, \ldots,\xi_n= \partial/\partial x_n$. That is,
\begin{gather*} \D =\bigg\{ f(\partial/\partial x_1, \ldots, \partial/\partial x_n) \,|\, f = \sum_{\alpha = (a_1, \ldots, a_n)} c_\alpha \xi_1^{a_1} \cdots \xi_n^{a_n} \in \Q[\xi_1, \ldots, \xi_n]\bigg\}.
\end{gather*}

Now let $I_P$ be the ideal of all differential operators $D \in \D$ such that $D (P) = 0$, i.e., those differential operators that annihilate $P$. We define $A_P$ to be the quotient algebra $\D / I_P$.
The algebra $A = A_P$ has a natural grading $A= \bigoplus_{i=0}^n A_i$ (this is because $I_P$ is a homogenous ideal). It follows from the construction that $A = A_P$ has the following properties:
\begin{itemize}\itemsep=0pt
\item $A_0=\Q$.
\item $A_k=0$ for $k>n$.
\item $\dim _{\Q}A_n=1$. Moreover $P$ defines a linear function on homogeneous order $n$ operators $D_n$ sending $D_n$ to the constant $D_n(P)$. It induces a linear isomorphism $A_n \to \Q$.
\item $A_1$ coincides with the image of $V$ in $A=\D / I_P$.
\item $A$ is generated as an algebra by $A_0$ and $A_1$.
\item One can define the pairing $B_k\colon A_k\times A_{n-k}\rightarrow \Q$ by $B_k(a_k, a_{n-k}) = D_k \circ D_{n-k}(P)$ where~$D_k$ and~$D_{n-k}$ are any preimages of the elements $a_k,a_{n-k}\in \D/I_P$ in~$D$. This pairing is well-defined since the preimages $D_k$ and $D_{n-k}$ are defined up to addition of elements from the ideal~$I_P$ and elements from~$I_P$ annihilate~$P$. It is easy to show that the pairing is non-degenerate and gives a ``Poincar\'e duality'' on $A$.
\end{itemize}

One has the following.
\begin{Th} \label{th-ring-of-diff-op}
Let $A = \bigoplus_{i=0}^n A_i$ be a graded algebra over $\Q$ with the following properties:
\begin{itemize}\itemsep=0pt
\item[$(1)$] $A_0=\Q$.
\item[$(2)$] There is a linear isomorphism $f\colon A_n \to \Q$.
\item[$(3)$] $A$ is generated as an algebra by $A_0$ and $A_1$.
\item[$(4)$] There is a linear projection $\pi\colon V\to A_1$ of a given $\Q$-vector space $V$ onto $A_1$.
\item[$(5)$] For any $0\leq k\leq n$ the pairing $B_k\colon A_k \times A_{n-k} \to \Q$ defined by $B_k(a_k,a_{n-k})=f(a_ka_{n-k})$ is non-degenerate.
\end{itemize}
Then $A$ can be described as follows:
\begin{itemize}\itemsep=0pt
\item[$(a)$] $A$ is isomorphic to the algebra associated to $(V, P)$ where $P$ is the homogeneous degree $n$ polynomial on $V$ defined by
\begin{gather*} P(x)=(1/n!) f\big(y^n\big),\end{gather*}
where $x\in V$ and $y=\pi(x)$.
\item[$(b)$] $A$ is isomorphic to the algebra associated to $(V, F)$ where~$F$ is the symmetric $n$-linear function on~$V$ equal to the polarization of the polynomial~$P$ from~$(a)$ multiplied by~$n!$.
\end{itemize}
\end{Th}

A generalization of Theorem \ref{th-ring-of-diff-op} for commutative algebras $A$ with Poincar\'{e} duality that are not necessarily generated by~$A_0$ and~$A_1$ can be found in \cite{EstKhKaz}. Also some related material can be found in \cite[Exercise~21.7]{Eisenbud} and \cite{Kaveh-note}.

\section[A description of ring of complete intersections of $G/H$]{A description of ring of complete intersections of $\boldsymbol{G/H}$} \label{sec-descp-ring-comp-int}

This section contains the main result of the paper which gives descriptions of the ring of complete intersections of a spherical homogeneous space in terms of volumes of polytopes. It is an analogue of the description of ring of conditions of torus~$(\C^*)^n$ in~\cite{Kaz-Khov}.

Recall $\mathcal{R}'(G/H)$ denotes the ring of complete intersections where $G/H$ is a spherical homogeneous space. We use Theorem~\ref{th-ring-of-diff-op} to give two descriptions of the ring $\mathcal{R}'_\Q(G/H) = \mathcal{R}'(G/H) \otimes_{\Z} \Q$ in terms of volumes of polytopes. Let $\Div(G/H)$ be the group of divisors on~$G/H$ with $\Div_\Q(G/H) = \Div(G/H) \otimes_\Z \Q$ the corresponding $\Q$-vector space. For an effective divisor~$D$ in~$G/H$ consider the associated Newton--Okounkov polytope~$\Delta(D)$ introduced in Section~\ref{sec-BKK-spherical}.

\begin{Th}The ring $\mathcal{R}'_\Q(G/H)$ is isomorphic to the algebra associated to the vector space $V = \Div_\Q(G/H)$ and the $n$-linear function $F\colon V \times \cdots \times V \to \Q$ whose value on an $n$-tuple $D_1, \ldots, D_n$ of effective divisors in $G/H$ is given by
$$(-1)^n F(D_1, \ldots, D_n) = - \sum_{i} \vol_n(\Delta_i) + \sum_{i < j} \vol_n(\Delta_{i,j}) + \cdots
+ (-1)^n \vol_n(\Delta_{1, \ldots, n}).$$ Here for $I \subset \{1, \ldots, n\}$ we set $\Delta_I =\Delta\big(\sum_{i \in I} D_i\big)$. Moreover, if $D_1, \ldots, D_n$ lie in a cone in $\Div_\Q(G/H)$ such that $D \mapsto \Delta(D)$ is linear then $F(D_1, \ldots, D_n) = n! V(\Delta(D_1), \ldots, \Delta(D_n))$ where~$V$ is the mixed volume $($see Section~{\rm \ref{sec-NO-polytope-sph-var})}.
\end{Th}
\begin{proof}The theorem follows from Theorem \ref{th-main}, Corollaries~\ref{cor-main} and~\ref{cor-mixed-vol-Brion-Kaz}, and Theo\-rem~\ref{th-ring-of-diff-op}.
\end{proof}

Similarly, we have the following description of $\mathcal{R}'_\Q(G/H)$ as a quotient of the ring of dif\-fe\-rential operators.

\begin{Th} \label{th-ring-complete-intersec-vol-poly}
The ring $\mathcal{R}'_\Q(G/H)$ is isomorphic to the algebra associated to the vector spa\-ce~$V = \Div_\Q(G/H)$ and the polynomial~$P$ whose value on an effective divisor~$D$ is equal to~$\vol_n(\Delta(D))$.
\end{Th}
\begin{proof}The theorem follows from Theorem \ref{th-main}, Corollary \ref{cor-main} and Theorem~\ref{th-ring-of-diff-op}.
\end{proof}

Finally, the ring of complete intersections $\mathcal{R}'(G/H)$ itself can be realized as the subring of~$\mathcal{R}'_\Q(G/H)$ generated by~$\Z$ and the image of~$\Div(G/H)$.

\section{Ring of complete intersections of an arbitrary variety}\label{sec-ring-complete-intersec-general}

Let $X$ be a variety (not necessarily complete). In \cite{KKh-MMJ, KKh-CMB} the authors consider the collection $\K(X)$ of all finite-dimensional vector subspaces of the field of rational functions $\C(X)$. This collection is a semigroup with respect to the multiplication of subspaces. Namely, for $L, M \in \K(X)$ we define
\begin{gather*} LM = \text{span}\{fg \,|\, f \in L, g \in M\}.\end{gather*}
We then consider the Grothendieck group $\G(X)$ associated to the semigroup of subspaces $\K(X)$.

Given $L_1, \ldots, L_n \in \K(X)$, one introduces {\it intersection index} $[L_1, \ldots, L_n] \in \Z$.
The intersection index $[L_1, \ldots , L_n]$ is the number of solutions $x \in X$ of a generic system of equations $f_1(x) = \cdots = f_n = 0$, where $f_i \in L_i$. In counting the solutions, we neglect the solutions $x$ at which all the functions in some space $L_i$ vanish as well as the solutions at which at least one function from some space $L_i$ has a pole. In \cite{KKh-MMJ, KKh-CMB} it is shown that the intersection index is well-defined and is multi-additive with respect to multiplication of subspaces. Thus it naturally extends to the Grothendieck group $\G(X)$ and to $\G_\Q(X) = \G(X) \otimes_\Z \Q$ the corresponding $\Q$-vector space. In \cite[Section~4.3]{KKh-Annals} we associate to each element $L\in \K(X)$ its Newton--Okounkov body $\Delta(L) \subset \R^n$ in such a way that the following conditions hold:
\begin{enumerate}\itemsep=0pt
\item[1)] $\Delta (L_1)+\Delta (L_2) \subset \Delta (L_1 L_2)$,
\item[2)] $[L,\dots,L]=n! \vol_n(\Delta(L)$.
\end{enumerate}

We view the intersection index of subspaces of rational functions as a birational version of the intersection theory of (Cartier) divisors (more generally linear systems).
The main result in \cite{KKh-CMB} is that the Grothendieck group~$\G(X)$ is naturally isomorphic to the group of b-divisors (introduced by Shokurov). Roughly speaking, a b-divisor is an equivalence class of Cartier divisors on any birational model of~$X$. The isomorphism sends the intersection index to the usual intersection number of Cartier divisors.

\begin{Def}[ring of complete intersections] \label{def-ring-of-complete-intersec-general}
We call the ring associated to the $\Q$-vector space $\G_\Q(X)$ and the intersection index, the {\it ring of complete intersections} over~$\Q$ of~$X$ and denote it by $\A_\Q(X)$. We call the subring $\A(X)$ of $\A_\Q(X)$ generated by~$\Z$ and the image of~$\G(X)$, the {\it ring of complete intersections} of~$X$.
\end{Def}

As above Theorem~\ref{th-ring-of-diff-op} suggests two descriptions of the ring $\A(X)$ of complete intersections of~$X$ in terms of volumes of Newton--Okounkov bodies associated to elements of~$\G(X)$.

\subsection*{Acknowledgements}
We are grateful to the anonymous referee for careful reading of the paper and several invaluable suggestions and corrections. The first author is partially supported by a National Science Foundation grant (grant ID: 1601303). The second author is partially supported by the Canadian grant No.~156833-12.

\pdfbookmark[1]{References}{ref}
\LastPageEnding

\end{document}